\documentclass{amsart}

\usepackage{amsmath,amssymb,amsfonts}
\usepackage[english]{babel}
\usepackage{amsxtra}
\usepackage{mathtools}
\usepackage[all]{xy}
\usepackage[dvipsnames]{xcolor}
\usepackage[colorlinks=true]{hyperref}
\usepackage{mathrsfs}

\title[Cuspidal cyclic cubic fourfold]
{The Fano variety of lines\\
of a cuspidal cyclic cubic fourfold}

\author{Samuel Boissi\`ere}

\address{Samuel Boissi\`ere,
	Laboratoire de Math\'ematiques et Applications,
	UMR 7348 du CNRS,
	B\^atiment H3,
	Boulevard Marie et Pierre Curie,
	Site du Futuroscope,
	TSA 61125,
	86073 Poitiers Cedex 9,
	France}

\email{samuel.boissiere@univ-poitiers.fr}

\urladdr{http://www-math.sp2mi.univ-poitiers.fr/$\sim$sboissie/}

\author{Tobias Heckel}

\address{Tobias Heckel,
		Gottfried Wilhelm Leibniz Universit{\"a}t Hannover, Institut f{\"u}r Algebraische Geometrie, Welfengarten 1, 30167 Hannover, Germany}

\email{heckel@math.uni-hannover.de}

\author{Alessandra Sarti}

\address{Alessandra Sarti,
	Laboratoire de Math\'ematiques et Applications,
	UMR 7348 du CNRS,
	B\^atiment H3,
	Boulevard Marie et Pierre Curie,
	Site du Futuroscope,
	TSA 61125,
	86073 Poitiers Cedex 9,
	France}

\email{alessandra.sarti@univ-poitiers.fr}

\urladdr{http://www-math.sp2mi.univ-poitiers.fr/$\sim$sarti/}

\date{\today}

\newtheorem{theorem}{Theorem}[section]
\newtheorem*{theoremintro}{Theorem}
\newtheorem{lemma}[theorem]{Lemma}

\newtheorem{corollary}[theorem]{Corollary}
\theoremstyle{definition}

\theoremstyle{remark}
\newtheorem{remark}{Remark}[section]

\numberwithin{equation}{section}

\newcommand{\IC}{\mathbb C}

\newcommand{\IP}{\mathbb P}

\newcommand{\IZ}{\mathbb Z}

\newcommand{\cC}{\mathscr{C}}

\newcommand{\cH}{\mathcal{H}}

\newcommand{\cM}{\mathcal{M}}

\newcommand{\cO}{\mathcal{O}}
\newcommand{\cP}{\mathcal{P}}

\newcommand{\cS}{\mathcal{S}}

\newcommand{\cX}{\mathcal{X}}
\newcommand{\cZ}{\mathcal{Z}}

\newcommand{\HH}{\mathrm{H}}
\newcommand{\BL}{\mathrm{Bl}}

\newcommand{\ta}{\tilde a}

\newcommand{\node}{\vartheta} 
\newcommand{\cusp}{\varsigma}

\newcommand{\eqY}{F}  %equation of the cubic fourfold Y
\newcommand{\eqC}{f}  %equation of the cubic threefold C

\newcommand{\NS}[1]{\mathrm{NS}(#1)}
\newcommand{\Pic}[1]{\mathrm{Pic}(#1)}
\newcommand{\Fano}[1]{\mathrm{F}(#1)}
\newcommand{\Jac}{\mathrm{J}}
\newcommand{\Sym}[2]{\mathrm{Sym}^{#1} \left({#2}\right)}
\newcommand{\Aut}[1]{\mathrm{Aut}(#1)}
\newcommand{\Hilb}[2]{{#2}^{[#1]}}
\newcommand{\Blow}[2]{\mathrm{Bl}_{#1}(#2)}
\newcommand{\Grass}[2]{\mathrm{Gr}(#1, \IP^{#2})}

\newcommand{\Span}{\mathrm{Span}}

\newcommand{\reg}{\mathrm{reg}}

\DeclareMathOperator{\orth}{O}
\DeclareMathOperator{\id}{id}

\begin{document}

\begin{abstract}
	We prove that the Fano variety of lines of a cuspidal cyclic cubic fourfold is a symplectic variety with transversal $A_2$-singularities and we study the properties of the nonsymplectic order three automorphism induced by the covering automorphism on the irreducible holomorphic symplectic manifold obtained by blowing up the singular locus.
\end{abstract}


\maketitle

%\tableofcontents

%%%%%%%%%%%%%%%%%%%%%%%%%%%%%%%%%%%%%%%%%%%%%%%%%%%%%%%%%%%%%%%%%%%%%%%%%%%%%%

\section{Introduction}

The classification of the automorphisms and birational transformations of irreducible holomorphic (IHS) manifolds is an intensive area of research that uses  the global Torelli theorem of Markmann--Verbitsky~\cite{Markman_survey, Verbitsky}  and the recent results of Bayer--Macr\`i~\cite{BM} to describe the nef and movable cones. Cuspidal cyclic cubic fourfolds  are constructed as triple coverings of nodal cubic threefolds. Their Fano varieties of lines are important since their periods are limit points under nodal degenerations of smooth cubic threefolds. They have been studied in~\cite{BCScubic} in the framework of the construction of moduli spaces, period maps and period domains of nonsymplectic automorphisms on IHS manifolds deformation equivalent to the Hilbert square of a K3 surface, as a hyper-K\"ahler interpretation of the theory developed by Allcock--Carlson--Toledo~\cite{ACT}. The meaning of the degeneracy of the automorphism is that when the period point
goes to the closure of the period domain, the automorphism of the family jumps to another family with a bigger invariant lattice. In~\cite{BCSball, BCScubic} the limit family is constructed using Hilbert squares of K3 surfaces with natural automorphisms, following the classification of Artebani--Sarti~\cite{AS}.
 
Starting from a nodal cubic threefold $C\subset\IP^4$, the triple covering of $\IP^4$ ramified along $C$ defines a cuspidal cyclic cubic fourfold $Y$. We denote by~$\sigma$ the automorphism of its Fano variety of lines $\Fano Y$ induced by the covering automorphism~$\iota$. The singular locus $\Sigma$ of $\Fano Y$ is a K3 surface with an automorphism~$\tau$ induced by~$\iota$. Theorem~\ref{th:Fano_cuspidal_cyclic_cubic} is the main result of this note. We describe the singularities of the Fano variety of lines of a cuspidal  cyclic cubic fourfold, its symplectic resolution and the corresponding divisorial contraction
of the Hilbert square of the underlying K3 surface. We are mostly interested in the behavior of the automorphisms induced by the covering automorphism:

\begin{theoremintro}{(Theorem~\ref{th:Fano_cuspidal_cyclic_cubic})}
	Let $Y$ be a cuspidal cyclic cubic fourfold with one cusp~$\cusp$.
	\begin{enumerate}
		\item  There exists a birational map $\varphi \colon \Hilb 2 \Sigma \dashrightarrow \Fano Y$ which commutes with the actions of $\Hilb 2 \tau$ and $\sigma$. If $Y$ contains no plane through $\cusp$, then $\varphi$ is everywhere defined and it contracts a divisor~$\Psi$ to the surface~$\Sigma$.
		
		\item  The variety $\Fano Y$ has symplectic singularities, they are $A_2$-transversal along the surface $\Sigma$ and $\Blow \Sigma {\Fano Y}$ is a symplectic resolution of $\Fano Y$. 
		
		\item Assume that $Y$ contains no plane through $\cusp$. The blowup $\rho$ of $\Sigma$ in~$\Fano Y$ is an elimination of the indeterminacies of the rational map $\varphi^{-1}$:
		\[
		\xymatrix{
			\Blow \Sigma {\Fano Y} \ar[d]^\rho \ar[dr]^{\widetilde{\varphi^{-1}}}  & \\
			\Fano Y \ar@{-->}[r]^-{\varphi^{-1}}                      & \Hilb 2 \Sigma
		}
		\]
		The morphism $\widetilde{\varphi^{-1}}$  is an isomorphism which maps isomorphically the exceptional divisor of the blowup to the prime divisor $\Psi$.
		
		\item  The automorphism $\sigma$ goes up to an order three nonsymplectic automorphism~$\widetilde\sigma$ on~$\Blow \Sigma {\Fano Y}$. If $Y$~contains no plane through~$\cusp$,
		the isomorphism~$\widetilde{\varphi^{-1}}$ commutes with the actions of $\widetilde\sigma$ and $\Hilb 2 \tau$.
	\end{enumerate}
\end{theoremintro}

In Remark~\ref{rem:fix_locus} we describe the fix locus of $\tilde\sigma$ and in \S\ref{ss:degeneration} we interpret the result in the context of nodal degenerations of cubic threefolds and holomorphic extensions of period maps: the pairs $\left(\BL_\Sigma\Fano {Y}, \tilde\sigma\right)$ and 
$\left(\Sigma^{[2]}, \tau^{[2]}\right)$ are equivariantly birational to each other and nonseparated in moduli space $\cM_L^\circ$ of $6$-polarized IHS manifolds, so both can be defined as limit points under a nodal degeneration of the underlying cubic threefold.

\medskip

To better understand the peculiarity of the birational geometry of the Fano variety of lines on a cuspidal cyclic fourfold, it is instructive to take a step backwards to study first the birational geometry of the Fano variety of lines on a nodal cubic fourfold, leaving apart for an instant the degeneration of automorphisms. Their geometry appears many times in the literature. Let us briefly recall two main reasons  why the nodal setup has become so popular.

First, it provides a nice example
of  \emph{symplectic singularity}, in
the sense defined by Beauville~\cite{Beauville_symplectic}:
its nonsingular locus admits a holomorphic symplectic form whose pullback extends to a holomorphic $2$-form
on any resolution of singularities of~$\Fano Y$, and it admits a \emph{symplectic resolution}, where the extended holomorphic $2$-form remains nondegenerate, obtained
by a single blowup of its singular locus, which is generically a K3 surface $\Sigma$. This has been first proved by Hassett~\cite[Lemma~6.3.1]{Hassett_special} by constructing a birational map from $\Hilb 2 \Sigma$ to $\Fano Y$. The result is recalled by Gounelas--Ottem~\cite[p.13]{GounelasOttem} and is generalized by Lehn~\cite[Theorem~3.6]{Lehn} to the case when the K3 surface $\Sigma$ is singular, by considering its minimal resolution $\widetilde\Sigma$.
Yamagishi~\cite{Yamagishi} further developed this study by proving that $\Fano Y$ has the same analytic type of singularity as the Hilbert scheme of two points on $\Sigma$ and that $\Fano Y$ has only one symplectic resolution up to isomorphism. Assuming that $Y$ has only simple singularities and no plane, he observed that the indeterminacy of the birational map $\widetilde\Sigma^{[2]}\to \Sigma^{[2]}$ can be resolved by a sequence of flops to produce a symplectic resolution of $\Sigma^{[2]}$, hence of $\Fano Y$ since the birational map $\Sigma^{[2]}\to\Fano Y$ is regular when $Y$ contains no plane.

Secondly, from the point of view of the  variety $\Hilb 2 \Sigma$, this setup provides an example of an irreducible holomorphic symplectic manifold with two divisorial contractions that generically describe its cone of numerically effective (nef) divisors. These contractions appear for instance in Hassett-Tschinkel~\cite{HassettTschinkel_moving, HassettTschinkel_intersection, HassettTschinkel_extremal}, Bayer--Macr\`i~\cite{BM} and Debarre--Macr\`i~\cite{DebarreMacri}.

\medskip

The geometry of the Fano variety of lines on a  cuspidal cyclic fourfold is more subtle than the nodal case, but it can be treated by a slight modification of the arguments. Since the nodal case is more familiar, we chose to develop the proof in the nodal case and to focus on the main changes needed in the cuspidal one. 
Our objective in Theorem~\ref{th:Fano_nodal_cubic} is to summarize all the geometric properties in the nodal setup, with emphasis on the properties of the birational maps involved, and to give a new and self-contained proof based on the stratification theorem of Kaledin~\cite{Kaledin_poisson} and using, as much as possible, only elementary geometric arguments. The benefit of our method is that it permits to have a deep control on the geometry of the blowup: this will be needed in the second part of this note to control the birational geometry and the automorphisms in the cuspidal cyclic setup considered in Theorem~\ref{th:Fano_cuspidal_cyclic_cubic}. 

As a byproduct, we recover in \S\ref{ss:divisorial}, using the geometric description used in the proof of Theorem~\ref{th:Fano_nodal_cubic} and direct geometric computations, the description of the extremal rays of the nef cone of $\Hilb 2 \Sigma$, in accordance with the results of~\cite{BM} (see also~\cite[Example~5.3]{DebarreMacri}).

The authors warmly thank Lucas Li Bassi, Chiara Camere, Bert van Geemen,  Franco Giovenzana, Luca Giovenzana, Klaus Hulek and Gianluca Piacenza for helpful discussions during the preparation of this work. The results of this paper are partially contained in the PhD thesis~\cite{Heckel} of the second author.

%After preparing this paper, we learnt that Tobias Heckel studied the same problem in his PhD %thesis~\cite{Heckel}. His techniques and results are broadly similar to ours.

\section{Geometry of the Fano variety of lines of a cubic fourfold}

\subsection{Definition} \label{ss:def_fano}
Let $Y \subset  \IP^5$ be a complex cubic fourfold, that is the zero locus of a degree three homogeneous polynomial $\eqY$. We denote by $\Fano Y$ the Fano scheme of lines contained in $Y$, considered as a closed subscheme of the Grassmannian $\Grass 1 5$ of lines in~$\IP^5$. 
For any line $\ell \subset \IP^5$, we denote by $|\ell]$ the corresponding point in $\Grass 1 5$. Set-theorerically:
\[
\Fano Y \coloneqq \{ [\ell] \in \Grass 1 5 \,|\, \ell\subset Y\} = \{[\ell] \in \Grass 1 5\,|\, \eqY\vert_{\ell} = 0\}.
\]
Following Altman--Kleiman~\cite{AK}, we define the scheme structure of $\Fano Y$ as follows. Let $V$ be a six-dimensional complex vector space. We put $\IP ^ 5 = \IP(V)$, the equation~$\eqY$ of~$Y$ is an element of the symmetric power $\Sym 3 {V^\ast}$. Consider the incidence variety:
\[
\cS\coloneqq \{([\ell], v) \in \Grass 1 5 \times V\,|\, v\in \ell\}.
\]
Any line $\ell \subset \IP^5$ is identified with the plane $\Pi_\ell \subset V$ such that $\ell = \IP(\Pi_\ell)$.
The fiber of the projection $\cS \to \Grass 1 5$ over the point $[\ell]$ is the plane $\Pi_\ell$ and this projection
makes $\cS$ a rank two vector bundle over $\Grass 1 5$.
The restriction of~$\eqY$ to the line~$\ell$ defines an element of $\Sym 3 {\Pi_\ell^\ast}$. We thus construct a regular section~$s_\eqY$ of the vector bundle $\Sym 3 {\cS^\ast}$:
\[
s_\eqY\colon \Grass 1 5 \to \Sym 3 {\cS^\ast}, \quad [\ell] \mapsto \eqY_{|\ell}.
\]
The Fano scheme $\Fano Y$ is defined as the subscheme of zeros of this section (see~\cite[Theorem~1.3]{AK}):
\[
\Fano Y \coloneqq \cZ(s_\eqY).
\]

\subsection{Singularities} \label{ss:singularities}

Throughout this paper, we consider a cubic fourfold $Y$ with one single simple isolated singularity of 
type either $A_1$ or $A_2$. 
In this geometric setup, by results of Altmann-Kleiman~\cite{AK}, the scheme $\Fano Y$ is a locally complete intersection,
it is reduced, normal, connected and irreducible (see also~\cite[Proposition~3.5]{Lehn}).
As a consequence, $\Fano Y$ has Gorenstein singularities.

The variety $\Fano Y$ is smooth at every point~$[\ell]$ corresponding to a line $\ell \subset Y$ 
which does not pass through its singular point. This is a general and classical result on Fano schemes, 
true in greater generality than ours (see for instance~ \cite[Corollary~1.11]{AK}, 
\cite[Lemma~7.7]{ClemensGriffiths} or \cite{Huybrechts}).  
Assuming that $Y$ does not contain a plane, the subspace $\Sigma \subset \Fano Y$ 
parametrizing those lines through  the singular point is a normal surface of degree~$6$ in~$\IP^4$, whose minimal resolution is a K3 surface. The surface $\Sigma$ is contained in the singular locus of $\Fano Y$ and, in good situations, it is itself smooth 
and it is exactly the singular locus of~$\Fano Y$ (see~\cite[Lemma~3.3]{Lehn}).

\section{The Fano variety of lines of a nodal cubic fourfold}

Let $Y$ be a cubic fourfold with one single ordinary double point $\node$.
In a suitable homogeneous coordinate system $(x_0 : \ldots : x_5)$ of 
$\IP^5$, the equation of~$Y$ may be written as:
\begin{align} 
	\label{eq:nodal_fourfold}
	\eqY(x_0, \ldots, x_5) = x_0 q(x_1, \ldots,x_5) + k(x_1, \ldots, x_5),
\end{align}
where $q$ is a nondegenerate quadratic form and is $k$ is a cubic form. The node~$\node$ has coordinates $[1:0:\ldots:0]$.
The zero loci of~$q$ and~$k$ define in the hyperplane~${H_0 \coloneqq \{x_0 = 0\}}$ of $\IP^5$ respectively a smooth quadric threefold~$Q$ and a cubic threefold~$K$.
Any line $\ell \subset Y$ passing through~$\node$ cuts the hyperplane~$H_0$ at a point of~$Q\cap K$. Since $Y$~has no other singularity than~$\node$, the intersection~$Q\cap K$ is nonsingular. In the sequel, we thus identify the locus $\Sigma \subset \Fano Y$ of those lines through~$\node$ with the K3 surface~$Q\cap K$.

The cubic fourfold $Y$ contains a plane passing through $\node$ if and only if the K3 surface $\Sigma$ contains a line. Indeed, if $Y$ contains a plane $P$ through~$\node$, then the lines through~$\node$ inside the plane~$P$ cut the hyperplane $H_0$ along a line which is contained in $\Sigma$ by construction. Conversely, if $\Sigma$ contains a line $\ell$, then the lines $\Span(x, \node)$ for $x\in \ell$ generate a plane contained in~$Y$. 
Degtyarev~\cite[Theorem~1.2]{Degtyarev} proved that the maximum number of lines on such K3 surfaces, or of planes in~$Y$ through $\node$, is the famous number $42$ (see~\cite{fortytwo}).

\subsection{Symplectic resolution of $\Fano Y$}

\begin{theorem} Let $Y$ be a cubic fourfold with one single ordinary double point $\node$.
	\label{th:Fano_nodal_cubic}
	\begin{enumerate}
		\item  \label{i:nodal1} There exists a birational map $\varphi \colon \Hilb 2 \Sigma \dashrightarrow \Fano Y$. If $Y$ contains no plane through $\node$, then $\varphi$ is everywhere defined and it contracts a divisor~$\Psi$ to the surface~$\Sigma$.
		
		\item  \label{i:nodal2} The variety $\Fano Y$ has symplectic singularities, they are $A_1$-transversal along the surface $\Sigma$ and $\Blow \Sigma {\Fano Y}$ is a symplectic resolution of $\Fano Y$.
		
		\item \label{i:nodal3} Assume that $Y$ contains no plane through $\node$. The blowup $\rho$ of $\Sigma$ in~$\Fano Y$ is an elimination of the indeterminacies of the rational map $\varphi^{-1}$:
		\[
		\xymatrix{
			\Blow \Sigma {\Fano Y} \ar[d]^\rho \ar[dr]^{\widetilde{\varphi^{-1}}}  & \\
			\Fano Y \ar@{-->}[r]^-{\varphi^{-1}}                      & \Hilb 2 \Sigma
		}
		\]
		The morphism $\widetilde{\varphi^{-1}}$ is an isomorphism which maps isomorphically the exceptional divisor of the blowup to the prime divisor $\Psi$, which is a conic bundle over $\Sigma$.
	\end{enumerate}
\end{theorem}

\begin{proof}\textit{}
	
	\noindent \textit{Proof of assertion~\eqref{i:nodal1}}.  We construct the birational map $\varphi \colon \Hilb 2 \Sigma \dashrightarrow \Fano Y$ as follows. First take two distinct points $x_1, x_2 \in \Sigma$. By definition of $\Sigma$, the lines $\ell_i = \Span(\node, x_i)$, $i=1, 2$ are contained in $Y$. If the plane $P \coloneqq \Span(\ell_1, \ell_2)$ is not contained in $Y$, it cuts~$Y$ along a residual line $\ell_3$ and we define $\varphi(\{x_1, x_2\}) = [\ell_3]$. Note that $\ell_3$ may go through the node $\node$. Now take a reduced subscheme of length two on~$\Sigma$, determined by a closed point $x \in \Sigma$ and a 
	tangent direction $v\in T_x\Sigma = T_xQ\cap T_x K$. Denote by $\ell_0 = \Span(\node, x)$ the line contained in $Y$, by~$\ell_1$ the line defined by $x$ and $v$, and consider the plane $P \coloneqq \Span(\ell_0, \ell_1)$. Assuming that $P$ is not contained in $Y$, let us prove that $\ell_0$~is a double line for the plane~$P$.  We may choose coordinates such that~$\ell_0$ has equations $x_2 = \cdots = x_5 = 0$. Equation~\eqref{eq:nodal_fourfold} thus decomposes as follows:
	\begin{equation}
		\label{eq:nodal_fourfold_precise}
		\begin{aligned}
			q &= x_1 h_1(x_2, x_3, x_4, x_5) + q_1(x_2, x_3, x_4, x_5), \\
			k &= x_1 ^  2 h_2(x_2, x_3, x_4, x_5) + x_1 q_2(x_2, x_3, x_4, x_5) + k_1(x_2, x_3, x_4, x_5),
		\end{aligned}
	\end{equation}
	where $h_1, h_2$ are linear forms, $q_1, q_2$ are quadratic forms and $k_1$ is a cubic form.
	We may further assume that $\ell_1$ has equation $x_0 = x_3 = x_4 = x_5 = 0$. 
	We denote by~$\hat x$ the projection of $x$ to $H_0$ by the point $\node$. In the hyperplane $H_0$, we thus have the tangent space $T_{\hat x} Q = \ker(h_1)$ and $T_{\hat x} K = \ker(h_2)$. In particular, since $Q$ and $K$ meet transversally, $h_1$ and $h_2$ are not proportional. By assumption, $\ell_1\subset T_{\hat x}Q \cap T_{\hat x} K$, so $h_1$ and $h_2$ vanish along $\ell_1$: we deduce that their equations do not depend on the variable $x_2$. The plane $P$ has equation $x_3 = x_4 = x_5 = 0$. Since $P$ is not contained in $Y$, we get that the equation of $P\cap Y$ is nonzero and factorizes by $x_2^2$: this means that the plane $P$ cuts $Y$ along the double line $\ell_0$ and a residual line $\ell_2$, and we define $\varphi( \{x, v\}) = [\ell_2]$.
	
	Let us describe the birational inverse $\varphi^{-1}\colon \Fano Y \dashrightarrow \Hilb 2 \Sigma$. 
	Let $[\ell] \in \Fano Y \setminus\Sigma$. Consider the quadratic cone $\widehat Q \subset \IP^5$ with base $Q$ and vertex $\node$ and the plane $P \coloneqq \Span(\node, \ell)$. Either $P$ is contained in $\widehat Q$, or the intersection $P\cap \widehat Q$ is a plane conic singular at $\node$.
	
	\begin{itemize}
		\item First case : $P\cap \widehat Q$ is the union of two distincts lines $\ell_1, \ell_2$ through $\node$.
		Put $x_i = \ell \cap \ell_i$ and $\hat x_i \in H_0$ the projection of $x_i$ through $\node$, for $i = 1, 2$. We have $F(x_i) = 0$ and $q(\hat x_i) = 0$,  so $k(\hat x_i) = 0$. This means that the lines $\Span(\node, x_i)$ are contained in $Y$ and $\varphi^{-1}([\ell]) = \{\hat x_1, \hat x_2\}$.
		
		\item Second case : $P\cap \widehat Q$ is  a double line $\ell_0$. Similarly we put $x = \ell\cap \ell_0$
		and $\hat x$ denotes the projection of $x$ on $H_0$ by $\node$. As above, $\hat x \in \Sigma$. Let $\ell'$~be the intersection of $P$ with $H_0$: it is the projection of the line $\ell$ to $H_0$ by~$\node$. Similar computations as above convince that $\ell'\subset T_x\Sigma$. We thus put $\varphi^{-1}([\ell]) = \{x, \ell'\}$ as a length two subscheme of $\Sigma$.
		
		\item Third case : $P\subset \widehat Q$. For any point $x\in P$, $x\neq \node$, with projection $\hat x \in H_0$, the same argument as above shows that the line $\Span(\node, \hat x)$ is contained in~$Y$, so $P\subset Y$.
	\end{itemize}
	
	Assume now that $Y$ contains no plane through $\node$. The morphism $\varphi$ is an isomorphism between the open locus of $\Hilb 2 \Sigma$ of subschemes $\xi$ such that the plane $P = \Span(\node, \xi)$ cuts $Y$ along a third line that does not pass by $\node$, and the open locus $\Fano Y \setminus \Sigma$. It contracts to the surface $\Sigma$ the ``trident'' divisor:
	\begin{align}
		\label{trident_divisor}
		\Psi \coloneqq \{\xi \in \Hilb 2 \Sigma\,|\, \Span(\node, \xi) \cap Y \text{ consists in three lines through } \node\}.
	\end{align}
	We can also describe $\Psi$ as the set of those subschemes $\xi \in \Hilb 2 \Sigma$ such that the projective line $\ell_\xi\coloneqq\Span(\vartheta, \xi)\cap H_0$ generated by $\xi$ cuts $\Sigma$ in three points. Since $\ell_\xi$~already cuts the cubic $K$ in three points, this is equivalent to say that $\ell_\xi$~cuts the quadric~$Q$ in three points. Finally, $\Psi$~is the set of those~$\xi$ such that the line~$\ell_\xi$ is contained in~$Q$.
	
	Let us compute the fibres of the restricted morphism $\varphi \colon \Psi \to \Sigma$. Take $[\ell] \in \Sigma$ and choose coordinates as above so that equations~\eqref{eq:nodal_fourfold_precise} hold. Any plane containing~$\ell$ cuts the three-dimensional space plane $x_0 = x_1 = 0$ at a unique point $(0:0:a_2:a_3:a_4:a_5)$ with $a\coloneqq (a_2:a_3:a_4:a_5)\in \IP^3$ and we denote this plane by~$P_a$. 
	The intersection $P_a\cap Y$ is the plane cubic in $\IP^2$ with coordinates $(t_0:t_1:t_2)$, given by the equation $F(t_0, t_1, t_2 a) = 0$. 
	The fiber of $\varphi$ over $[\ell]$ is the locus of points $a\in \IP^3$ that parametrizes those planes $P_a$ such that $P_a\cap Y$ is the union of three non-necessarily distinct lines though~$\node$. Let us show that this locus is a nonsingular plane conic.
	The line $\ell$ has equation $t_2 = 0$ in this plane and we compute that the residual conic~$C_a$ has equation:
	\begin{align}
		\label{eq:residual_conic}
		h_1(a) t_0 t_1  + q_1(a) t_0 t_2  + h_2(a) t_1 ^ 2  + q_2(a)  t_1 t_2 + k_1(a) t_2 ^ 2 = 0.
	\end{align}
	The fibre of $\varphi$ over $[\ell]$ is described by those planes $P_a$ whose singular residual conic~$C_a$ is the union of two lines through $\node$. This is equivalent to the conditions:
	\begin{align}
		\label{eq:condition_residual}
		h_1(a) = q_1(a) = 0.
	\end{align}
	This defines a plane conic $\varphi^{-1}([\ell])$. Let us check that it is nonsingular. Assume that a point~$a \in \IP^3$ is a singular point of $\varphi^{-1}([\ell])$. Then the gradients vectors $\nabla h_1$ and $\nabla q_1(a)$ are proportional, but $\nabla h_1 \neq 0$ otherwise $q$~would be degenerate by Equation~\eqref{eq:nodal_fourfold_precise}, so there exists $\alpha\in \IC$ such that $\nabla q_1(a) = \alpha \nabla h_1$. Using the equation $q = x_1 h_1 + q_1$, we get $q(-\alpha, a) = 0$ and $\nabla q(-\alpha, a) = 0$: this is impossible since $q$~is nondegenerate. We conclude that the fibres of~$\varphi$ are isomorphic to nonsingular plane conics of equation~\eqref{eq:condition_residual} whose points~$a$ parametrize those planes~$P_a$ such that $P_a\cap Y$~is the union of three non-necessarily distinct lines though~$\node$.
	
	\medskip
	\noindent \textit{Proof of assertion~\eqref{i:nodal2}}. Let us first prove that $\Fano Y$ has symplectic singularities, using an argument close to~\cite[Theorem~3.6]{Lehn}.
	 By assertion~\eqref{i:nodal1},  the map $\varphi^{-1}$ is regular and bijective on the open subset $U$ of $\Fano Y^\reg \coloneqq \Fano Y \setminus \Sigma$ consisting of those lines not contained in any plane of $Y$ passing through its node. For each such plane, the locus of these lines form a $\IP^1$ in $\Fano Y$, so the complementary $Z$ of~$U$ in~$\Fano Y^\reg$ has codimension three. The open subset~$U$ is isomorphic to an open subset~$V$ of~$\Hilb 2 \Sigma$.  In particular, the canonical bundle of $U$ is trivial, hence $K_{\Fano Y^\reg} =  0$. Since $\Fano Y$ is irreducible and normal, this implies that  $K_{\Fano Y} = 0$ as a Cartier divisor. 
	 
	Consider an elimination of the indeterminacies of the birational map $\varphi^{-1}$. We have birational morphisms~$g, h$ and a commutative diagram:
	\[
	\xymatrix{ & W\ar[dl]_g \ar[dr]^h\ & \\ \Fano Y \ar@{-->}[rr]^{\varphi^{-1}} && \Hilb 2 \Sigma}
	\]
	where~$W$ can be further assumed to be nonsingular.
	The pullback by $h$ induces an isomorphism between the spaces of global sections of the canonical sheaves (see~\cite[III.6.1]{Shafarevich}):
	\[
	\HH^0(W, \omega_W) \cong \HH^0(\Hilb 2 \Sigma, \omega_{\Hilb 2 \Sigma}) = \HH^0(\Hilb 2 \Sigma, \cO_{\Hilb 2 \Sigma})  = \IC,
	\]
	so the canonical divisor $K_W$ is effective.
 We have trivially $K_W - g^\ast K_{\Fano Y} = K_W$ and we proved that $K_W$ is effective: by definition this means that $\Fano Y$ has canonical singularities. 
	By Elkik--Flenner theorem (see~\cite[\S3(C) \& p.363]{ReidYPG}) we deduce that $\Fano Y$ has rational singularities.
	
The open subset~$U$ is isomorphic to an open subset~$V$ of~$\Hilb 2 \Sigma$, so $U$~admits a symplectic form inherited from those of~$V$. 
	Since the sheaf of holomorphic two-forms $\Omega^2_{\Fano Y^\reg}$ is normal (see~\cite[Proposition~1.6]{Hartshorne_reflexive}), its sections do not depend on a closed subset $Z$ of codimension three, so $\HH^0(U, \Omega^2_U) = \HH^0(\Fano Y ^\reg, \Omega^2_{\Fano Y ^\reg})$. This shows that the symplectic form on~$U$ extends to a holomorphic two-form on $\Fano Y^\reg$, that is still closed since $\Fano Y$ has rational singularities (see Kebekus--Schnell~\cite[Theorem~1.13]{KS}) and nondegenerate otherwise it would degenerate along a divisor cutting also $U$.  By Namikawa theorem~\cite[Theorem~6]{Namikawa}, we deduce that $\Fano Y$ has symplectic singularities.
	
	By the stratification theorem of Kaledin~\cite[Theorem~2.3]{Kaledin_poisson}, as explained in Lehn--Mongardi--Piacenza~\cite[Proposition~2.2]{MLP} we deduce in our situation that $\Fano Y$~has transversal ADE singularities along the surface $\Sigma$: for any point $x\in \Sigma$, there exists an analytic neighborhood $U$ of $x$ in $\Fano Y$ that is isomorphic to $V\times S$, where $V$ is an open neighborhood
	of the origin in $\IC^2$ and $S$ is an ADE surface singularity. 
	It remains to determine the nature	of the singularity of $S$. This is a local computation, we use the notation of the proof of assertion~\eqref{i:nodal2}.
	By the Pl\"ucker embedding $\Grass 1 5 \hookrightarrow \IP^{14}$, the Grassmannian of lines in $\IP^5$ is locally isomorphic to the affine space $\IC^8$. The Pl\"ucker coordinates in an affine neighbourhood of the point $[\ell_0]$  characterize those lines passing through the points:
	\[
	(1 : 0 : - p_{1, 2} : - p_{1, 3} : - p_{1, 4} : - p_{1, 5}), 
	\quad 
	(0 : 1 : p_{0, 2} : p_{0, 3} : p_{0, 4} : p_{0, 5}).
	\]
	We put $p_j \coloneqq (p_{j, 2}, p_{j, 3}, p_{j, 4}, p_{j, 5})$ for $j = 0, 1$, so that such a line, that we denote by $\ell_{p_0, p_1}$, is parametrized by:
	\[
	x_0 = \lambda, \quad x_1 = \mu, 
	\quad 
	(x_2, x_3, x_4, x_5) = -\lambda p_1 + \mu p_0, \qquad \forall [\lambda:\mu] \in  \IP^1.
	\]
	By replacing in Equation~\eqref{eq:nodal_fourfold_precise} and extracting the coefficients in $\lambda^i\mu^{3-i}$, we get that $\Fano Y$~is defined in this chart $\IC^4\times \IC^4$ of coordinates $(p_0, p_1)$ as the zero locus $Z$ of the following four equations:
	\begin{equation}
		\label{eq:equations_fano_nodal}
		\begin{aligned}
			\psi_{3,0} &=q_1(p_1) + k_1(p_1), \\
			\psi_{0, 3} &= h_2(p_0) + q_2(p_0) + k_1(p_0), \\
			\psi_{2, 1} &= - h_1(p_1)  - 2b_1(p_1, p_0) + q_2(p_1) + k_1^{2,1}(p_1, p_0),\\
			\psi_{1, 2} &= h_1(p_0) - h_2(p_1) + q_1(p_0) - 2b_2(p_1, p_0) + k_1^{1,2}(p_1, p_0),
		\end{aligned}
	\end{equation}
	where $b_i$ are the bilinear forms associated to the quadratic forms $q_i$ and $k_1^{u, v}$ are the terms of bidegree $(u, v)$ in $k_1$ considered here as a cubic form in the variables~$\lambda, \mu$.
	In this chart, the surface $\Sigma \cap Z$ has equation $p_1 = 0$.
	The Jacobian matrix of $\Fano Y$ at the point $[\ell_0]$ of coordinates $p_0 = p_1 = 0$ is thus:
	\[
	\Jac_{\Fano Y}([\ell_0]) 
	=
	\begin{pmatrix}
		0 & \nabla {h_2} & 0              & \nabla {h_1} \\
		0 & 0            & - \nabla {h_1} & - \nabla {h_2}
	\end{pmatrix}.
	\]
	As we observed above, since $Q$ and $K$ meet transversally, the linear form $h_1$ and~$h_2$, or equivalently their gradients vectors $\nabla h_1$ and $\nabla h_2$, are not proportional. 
	By a linear change of variables in the variables $x_2, \ldots, x_5$ only, we may thus assume, without loss of generality, that $h_1(x_2, \ldots, x_5) = x_2$ and $h_2(x_2, \ldots, x_5) = x_3$. The Jacobian submatrix of the functions $\psi_{0, 3}, \psi_{1, 2}$ with respect to the variables $p_{0, 2}, p_{0, 3}$ is thus the identity matrix. By the holomorphic implicit function theorem, we may thus express the variables $p_{0, 2}, p_{0, 3}$ locally at the origin as holomorphic functions $\tilde p_{0, 2}, \tilde p_{0, 3}$ in the variables 
	$p_{0, 4}, p_{0, 5}, p_{1, 2}, p_{1, 3}, p_{1, 4}, p_{1, 5}$. In our affine chart, the variety~$\Fano Y$ is thus locally biholomorphic,  for $p_0 \in \IC^4$ in a neighborhood of the origin, to the subvariety $\cX \subset \IC^2\times \IC^4$ of coordinates $p_{0, 4}, p_{0, 5}, p_{1, 2}, p_{1, 3}, p_{1, 4}, p_{1, 5}$ given by the following two holomorphic equations:
	\begin{align*}
		\widetilde\psi_{3, 0}(p_1) &= q_1(p_1) - k_1(p_1),\\
		\widetilde\psi_{2, 1}(p_0, p_1) &=  - h_1(p_1) - 2b_1(p_1, p_0) + q_2(p_1) 
		+ k_1^{2,1}(p_1,p_0),
	\end{align*}
	where in the second equation, the coordinates $p_{0, 2}$ and $p_{0, 3}$ are
	replaced by their holomorphic expressions  $\tilde p_{0, 2}, \tilde p_{0, 3}$ in terms of the other variables.
	The coordinates $(p_{0, 4}, p_{0, 5})$ are local coordinates of $\Sigma$ at $[\ell_0]$, so over the origin $(p_{0, 4}, p_{0, 5}) = (0, 0)$ we have $p_0 = 0$. Since we know that $\Fano Y$ has transversal singularities along $\Sigma$, the type of the singularity is those of the fibre over the origin of the projection $\cX \to \IC^2$ to the coordinates $p_{0, 4}, p_{0, 5}$. This fibre is the surface~$\cX_0$ of~$\IC^4$ of coordinates~$p_1$ given by the two equations:
	\begin{align*}
		\overline\psi_{3, 0}(p_1) &= q_1(p_1) - k_1(p_1),\\
		\overline\psi_{2, 1}(p_1) &=  - h_1(p_1) + q_2(p_1).
	\end{align*}
	We observed in the proof of assertion~\eqref{i:nodal1} that the zero loci of $h_1$ and $q_1$ intersect transversally, so the restriction of the quadratic form $q_1$ to the hyperplane $\ker(h_1)$ has rank three. This means that $\cX_0$ has an $A_1$-singularity at the origin. We thus proved that~$\Fano Y$
	has transversal $A_1$-singularities along~$\Sigma$. The blowup $\Blow \Sigma {\Fano Y}$ is thus a crepant resolution of $\Fano Y$ (see for instance Perroni~\cite[Proposition~4.2]{Perroni}), and this is equivalent to be a symplectic resolution (see for instance~\cite[Proposition~1.6]{Fu}).

	\medskip	
	\noindent \textit{Proof of assertion~\eqref{i:nodal3}}. This is a local computation, we still use the notation of the proof of assertion~\eqref{i:nodal2} to compute the equations of the blowup of $\Sigma$ in an affine neighborhood of the line $\ell_0$ of equation $x_2 = \cdots = x_5 = 0$. Since $Z\cap \Sigma$ has equation $p_1 = 0$ (see Equation~\eqref{eq:equations_fano_nodal}), this blowup is given locally as the closure of the image of the regular morphism:
	\[
	Z \setminus (\Sigma \cap Z) \to Z \times \IP^3, 
	\quad (p_0, p_1) \mapsto \left( (p_0, p_1), (p_{1, 2} : p_{1, 3} : p_{1, 4} : p_{1, 5}) \right).
	\]
	We denote by $a \coloneqq (a_2 : a_3 : a_4 : a_5)$ the homogeneous coordinates of the $3$-dimensional projective space occuring in this blowup. Assuming for instance $a_5 \neq 0$, let us put 
	$\ta =\left(\frac{a_2}{a_5}, \frac{a_3}{a_5}, \frac{a_4}{a_5}, 1 \right)$. The blowup gives the relations $p_{1,i} = \frac{a_i}{a_5} p_{1, 5}$ and we make the change of variables $x = -\lambda p_{1, 5}\ta + \mu p_0$. Computing as above, we obtain the local equations of $\Blow \Sigma {\Fano Y}$ over $Z$:
	\begin{equation}
		\label{eq:equations_blow_fano_nodal}
		\begin{aligned}
			\overline\psi_{3, 0} &=  q_1(\ta) - p_{1, 5} k_1(\ta),\\
			\overline\psi_{0, 3} &= h_2(p_0) + q_2(p_0) + k_1(p_0),\\
			\overline\psi_{2, 1} &= -  h_1(\ta) - 2  b_1(\ta, p_0) + p_{1, 5} q_2(\ta) 
			+ p_{1, 5} k_1^{2,1}(\ta, p_0), \\
			\overline\psi_{1, 2} &= h_1(p_0) + q_1(p_0) - p_{1, 5} h_2(\ta) - 2 p_{1, 5} b_2(\ta, p_0)
			- p_{1, 5} k_1^{1,2}(\ta, p_0).
		\end{aligned}
	\end{equation} 
	The fibre of $\rho$ over $[\ell_0]$ is obtained by putting $p_0 = 0$ and $p_{1, 5} = 0$. We obtain after homogeneization:
	\[
	\rho^{-1}([\ell_0]) = \{a\in\IP^3\,|\, q_1(a) = h_1(a) = 0\}.
	\]
	We recover Equation~\eqref{eq:condition_residual}. This means that the blowup parametrizes the planes~$P_a$ containing~$\ell_0$ and such that the residual conic $C_a$ is the union of two lines by~$\node$. By the construction of $\varphi^{-1}$ explained in the proof of assertion~\eqref{i:nodal1}, the blowup thus eliminates the indeterminacies of $\varphi^{-1}$, since by assumption $Y$~contains no plane through $\node$. We obtain a morphism $\widetilde{\varphi^{-1}}$ which means geometrically that the coordinate $a$ selects one plane $P_a$ to determine the image length two subscheme on~$\Sigma$ by the construction explains in the proof of assertion~\eqref{i:nodal1}. The morphism $\widetilde{\varphi^{-1}}$ is thus birational and bijective, so it is an isomorphism by Zariski main theorem. By construction, it maps isomorphically the exceptional divisor of the blowup, which is a conic bundle over $\Sigma$, to the divisor $\Psi$.
\end{proof}

\begin{remark}
	Let $\ell \subset Y$ be a line passing through $\node$. The planes containing $\ell$ and such that the residual intersection with $Y$ is a degenerate conic (non necessarily passing through $\node$) are parametrized by a quintic surface $T$ in~$\IP^3$: this is a special case of a \emph{Togliatti surface}, the difference here is that the line $\ell$ is not generic (we refer to~\cite[\S2]{Supersac} for a recent and selfcontained synthesis on Togliatti surfaces). Using Equation~\eqref{eq:residual_conic}, we get that the equation of $T$ is:
	\[
	\det \begin{pmatrix}
		0 & h_1(a) & q_1(a)\\
		h_1(a) & 2 h_2(a) & q_2(a)\\
		q_1(a) & q_2(a) & 2 k_1(a)
	\end{pmatrix}
	= 0.
	\]
	It is a non-normal surface, whose singular locus is generically the curve of equations $h_1(a) = q_1(a) = 0$. The key geometric ingredients of the proof of Theorem~\ref{th:Fano_nodal_cubic} are the properties of the singular curve of this non-normal Togliatti surface.
	
\end{remark}

%\begin{remark}
%	Lehn~\cite[Theorem~3.6]{Lehn} uses the direct construction of a symplectic form on %$\Fano Y\setminus \Sigma$ given by De Jong--Starr and he wonders %in~\cite[Remark~3.7]{Lehn} whether the use of the birational model $\Hilb 2 \Sigma$ might %be avoided to prove that $\Fano Y$ has symplectic singularities. We give an explicit %example of this strategy in~\S\ref{s:explicit_nodal}:  on an easily computable example, %we show directly that $\Fano Y$ has transversal $A_1$-singularities, we take an explicit %symplectic form on $\Fano Y \setminus \Sigma$ and we compute that it extends to a %symplectic form on $\Blow \Sigma {\Fano Y}$. 
%\end{remark}

\subsection{Two divisorial contractions}
\label{ss:divisorial}

We  observed above two divisorial contractions of $\Hilb 2 \Sigma$, contracting a divisor to a locus isomorphic to
the K3 surface $\Sigma \subset \IP^4$, on two varieties with transversal $A_1$-singularities along $\Sigma$: the first one is the Chow quotient $\Sigma^{(2)}$, obtained by contracting the ``exceptional divisor'' $E$ of nonreduced subschemes, the second one is $\Fano Y$ obtained by contracting the ``trident'' divisor~$\Psi$ defined in Equation~\eqref{trident_divisor}. In particular $\Hilb 2 \Sigma$ is their unique crepant, hence symplectic resolution. These divisorial contractions have been studied by many authors from different point of views, see for instance \cite{BM, GounelasOttem,  HassettTschinkel_intersection, HassettTschinkel_extremal, KapustkaGeemen, Wierzba}. Our method gives a direct and geometric computation of the cone of numerically effective divisors of~$\Hilb 2 \Sigma$,
in accordance with the general results stated in the literature (see for instance in~\cite[Example~5.3]{DebarreMacri} and references therein), and it gives a concrete realization of the so-called ``second divisorial contraction''.

Consider the  generic situation where the N\'eron--Severi group of the K3 surface~$\Sigma$ is $\NS \Sigma = \IZ [H]$, with $H = \cO_\Sigma(1)$ of self-intersection $6$. We get:
\[
\NS {\Hilb 2 \Sigma}\cong \IZ h \oplus \IZ \delta,
\] 
where $[2\delta] = E$ and $h$ is the image of $[H]$ by the natural group homomorphism $\NS \Sigma \to \NS{\Hilb 2 \Sigma}$. For the Beauville--Bogomolov--Fujiki quadratic form on~$\NS{\Hilb 2 \Sigma}$ we have $\delta ^ 2 = -2$ and $h ^ 2 = 6$. 

\begin{lemma}
	The class of the divisor $\Psi$ in $\NS{\Hilb 2 \Sigma}$ is $[\Psi] = h - 2 \delta$.
\end{lemma}

\begin{proof}
	Choose a general curve $C\in |H|$ and a point $x\in \Sigma \setminus C$ and consider the curve~$\Gamma$ of~$\Hilb 2 \Sigma$ defined by the subschemes $\xi = (x, a)$ where $a$ varies in $C$. Clearly $\int_\Gamma h = 6$ and since $\Gamma\cap E = \emptyset$ we have $\int_\Gamma \delta = 0$. We may assume that $x$ has coordinates $(0:1:0:\ldots:0)$ and that $C$ has equation $x_1 = 0$, hence $\Gamma \cap \Psi$ is the locus of subschemes $\xi = (x, a)$ such that $a\in C$ and $\Span(\xi, \node)$ consists in three lines through~$\node$. We repeat the computation done after Equation~\eqref{trident_divisor}: the line  $\Span(\node, a)$ has equation $t_1 = 0$ and is contained in the residual conic $C_a$, so looking at Equation~\eqref{eq:residual_conic} we get $q_1(a) = k_1(a) = 0$. The locus $\Gamma \cap \Psi$ is thus the set of points $a\in C$ such that $h_1(a) = q_1(a) = k_1(a) = 0$, so $\int_\Gamma [\Psi] =  6$.
	Consider now the curve $\Lambda \coloneqq\IP(T_x\Sigma)\subset \Hilb 2 \Sigma$ parametrizing nonreduced subschemes supported at the point~$x$. Clearly $\int_{\Lambda} h = 0$ and the formula of Ellingsrud--Str{\o}mme~\cite[Theorem~1.1]{ES} gives $\int_{\Lambda} \delta = -1$. We compute $\Lambda \cap \Psi$ as above: this time $\Span(\node, x)$ is the double line of equation $t_2 = 0$ so $h_1(a) = h_2(a) = 0$ and the residual line $\Span(\node, a)$ imposes $q_1 = 0$, so $\int_\Lambda [\Psi] = 2$. It follows that $[\Psi] = h - 2\delta$.
\end{proof}

The Hilbert--Chow morphism $\Hilb 2 \Sigma\to \Sigma^{(2)}$ contracts the divisor $E$, whose orthogonal class in $\NS{\Hilb 2 \Sigma}$ is the ray $\IZ h$. The divisorial contraction $\varphi\colon\Hilb 2 \Sigma \to \Fano Y$ contracts the divisor $\Psi$, whose orthogonal 
class in $\NS{\Hilb 2 \Sigma}$ is, by a straighforward computation, the ray $\IZ (2h - 3\delta)$.
The Nef cone of $\Hilb 2 \Sigma$ is thus generated by the classes $h$ and $2h - 3\delta$, as stated in \cite[Example~5.3]{DebarreMacri} following Bayer--Macr\`i~\cite{BM}. A different interpretation of this divisorial contraction is given in~\cite[Example~5.4]{KapustkaGeemen}.

\section{The Fano variety of lines of a cuspidal cyclic cubic fourfold}

\label{section_fano_variety_lines_cuspidal_cubic_fourfold}

Let $C \subset \IP^4$ be a cubic threefold with one single ordinary double point $\node$.
In a suitable  homogeneous coordinate system $(x_0 : \ldots : x_4)$ of $\IP^4$, the equation of~$C$ may be written as:
\begin{align} 
	\label{eq:nodal_threefold}
	\eqC(x_0, \ldots, x_4) = x_0 q(x_1, \ldots,x_4) + g(x_1, \ldots, x_4),
\end{align}
where $q$ is a nondegenerate quadratic form and $g$ is a cubic form. The node~$\node$ has coordinates $[1:0:\ldots:0]$.
Consider the cubic fourfold $Y \subset \IP^5$ defined as the triple covering of $\IP^4$ branched along $C$ 
and denote by $\iota \in \Aut Y$ the covering automorphism. Using the same convention as above, the equation of~$Y$ may be written as:
\begin{align} 
	\label{eq:cuspidal_fourfold}
	\eqY(x_0, \ldots, x_5) = x_0 q(x_1, \ldots,x_4) + g(x_1, \ldots, x_4) + x_5^3,
\end{align}
and $\iota(x_0:\ldots:x_5) = (x_0:\ldots:x_4:\xi x_5)$, where $\xi$ is a primitive third root of unity.
The node $\node$ of $C$ induces an isolated singular point $\cusp$ of type $A_2$ on $Y$. We call
$Y$ a \emph{cyclic cuspidal cubic fourfold}.

We denote $k(x_0,\ldots,x_5) \coloneqq g(x_0, \ldots, x_4) + x_5^3$. The zero loci of~$q$ and~$k$ define in the hyperplane~$H_0 \coloneqq \{x_0 = 0\} \subset \IP^5$ respectively a quadratic cone of dimension three~$Q$ and a cubic threefold~$K$. Note that the vertex of $Q$ is not in $K$.

Any line $\ell \subset Y$ passing through~$\cusp$ cuts the hyperplane~$H_0$ at a point of~$Q\cap K$. Since $Y$~has no other singularity than~$\cusp$, the intersection~$Q\cap K$ is nonsingular. We identify the locus $\Sigma \subset \Fano Y$ of those lines through~$\cusp$ with the K3 surface~$Q\cap K$.
As in the nodal case, the cubic fourfold $Y$ contains a plane passing through $\cusp$ if and only if the K3 surface $\Sigma$ contains a line and the number of such lines is at most~$42$ (see \cite[Theorem~1.2]{Degtyarev}).

The homography $\iota$ of $\IP^5$ acts naturally of~$\Grass 1 5$ and restricts to an order three automorphism $\sigma$ of $\Fano Y$ which is nonsymplectic~\cite[Lemma~6.2]{BCSclass}. Similarly, $\iota$~produces a nonsymplectic order three automorphism~$\tau$ on~$\Sigma$ which is simply the restriction of the homography $\iota (x_1 :\ldots : x_5) = (x_1:\ldots: x_4 : \xi x_5)$ to the hyperplane~$H_0$. It induces a natural order three nonsymplectic automorphism~$\Hilb 2 \tau$ on~$\Hilb 2 \Sigma$.

\subsection{Symplectic resolution of $\Fano Y$}

If $Y$ has an equation as \eqref{eq:nodal_fourfold}, but where the quadratic form has rank four, we may assume that it does not depend on the variable $x_5$, so the point $\node$ defines a cusp in $Y$. In this setup, assuming that the K3 surface~$\Sigma$ is nonsingular, analogous statements as in Theorem~\ref{th:Fano_nodal_cubic} can be proven using the same lines, with the following changes: in assertion \eqref{i:nodal2} the variety~$\Fano Y$ has transversal $A_2$-singularities along~$\Sigma$ and in assertion~\eqref{i:nodal3} the fibres of the divisor~$\Psi$ over a point of~$\Sigma$ are the union of two rational curves intersecting transversally. We will state this result only in the more restrictive setup which is of interest for us in this note, where $Y$~is obtained as a cyclic covering, since we are mostly interested in the behaviour of the symmetries induced by the covering automorphism.

\begin{theorem} Let $Y$ be a cuspidal cyclic cubic fourfold with one cusp $\cusp$.
	\label{th:Fano_cuspidal_cyclic_cubic}
	\begin{enumerate}
		\item  \label{i:cusp1} There exists a birational map $\varphi \colon \Hilb 2 \Sigma \dashrightarrow \Fano Y$ which commutes with the actions of $\Hilb 2 \tau$ and $\sigma$. If $Y$ contains no plane through $\cusp$, then $\varphi$ is everywhere defined and it contracts a divisor~$\Psi$ to the surface~$\Sigma$.
		
		\item  \label{i:cusp2} The variety $\Fano Y$ has symplectic singularities, they are $A_2$-transversal along the surface $\Sigma$ and $\Blow \Sigma {\Fano Y}$ is a symplectic resolution of $\Fano Y$. 
		
		\item \label{i:cusp3} Assume that $Y$ contains no plane through $\cusp$. The blowup $\rho$ of $\Sigma$ in~$\Fano Y$ is an elimination of the indeterminacies of the rational map $\varphi^{-1}$:
		\[
		\xymatrix{
			\Blow \Sigma {\Fano Y} \ar[d]^\rho \ar[dr]^{\widetilde{\varphi^{-1}}}  & \\
			\Fano Y \ar@{-->}[r]^-{\varphi^{-1}}                      & \Hilb 2 \Sigma
		}
		\]
		The morphism $\widetilde{\varphi^{-1}}$  is an isomorphism which maps isomorphically the exceptional divisor of the blowup to the divisor $\Psi$.
		
		\item \label{i:cusp4} The automorphism $\sigma$ goes up to an order three nonsymplectic automorphism~$\widetilde\sigma$ on~$\Blow \Sigma {\Fano Y}$. If $Y$~contains no plane through~$\cusp$,
		the isomorphism~$\widetilde{\varphi^{-1}}$ commutes with the actions of $\widetilde\sigma$ and $\Hilb 2 \tau$.
	\end{enumerate}
\end{theorem}

\begin{proof}
	
	The proof follows the same lines as those of Theorem~\ref{th:Fano_nodal_cubic} : we thus only write the main changes. 
	
	\medskip
	\noindent\textit{In the proof of assertion~\eqref{i:cusp1},}
	Equation~\eqref{eq:nodal_fourfold_precise} becomes:
	\begin{equation}
		\label{eq:cupsidal_fourfold_precise}
		\begin{aligned}
			q &= x_1 h_1(x_2, x_3, x_4) + q_1(x_2, x_3, x_4), \\
			k &= x_1 ^  2 h_2(x_2, x_3, x_4) + x_1 q_2(x_2, x_3, x_4) + k_1(x_2, x_3, x_4) + x_5 ^ 3
		\end{aligned}
	\end{equation}
	and Equation~\eqref{eq:residual_conic}	becomes:
	\begin{align}
		\label{eq:residual_conic_cusp}
		h_1(a) t_0 t_1  + q_1(a) t_0 t_2  + h_2(a) t_1 ^ 2  + q_2(a)  t_1 t_2 + k_1(a) t_2 ^ 2 + a_5 ^ 3 t_2 ^ 2.
	\end{align}
	Equation~\eqref{eq:condition_residual} defines this time two lines meeting at $(0:0:0:1) \in \IP^3$. Let us check that the plane $\{h_1(a) = 0\}$ is not tangent to the quadratic cone $\{q_1(a) = 0\}$. If this occurs, then there exists a point $b\coloneqq (b_2:b_3:b_4)\in \IP^2$ such that $q_1(b) = h_1(b) = 0$ and $\nabla q_1(b) = \alpha \nabla h_1$ with $\alpha \neq 0$. Then $\nabla q(-\alpha, b, 0) = 0$ this is impossible since $Q$ nonsingular away from the point $(0:\ldots:0: 1) \in H_0$.
	We conclude that the fibres of $\varphi$ are two distinct meeting lines whose points~$a$ parametrize those planes $P_a$ such that $P_a\cap Y$ is the union of three non-necessarily distinct lines though~$\cusp$. Note that at the point $a = (0:\ldots:0 :1)$, the residual conic $C_a$ has equation $t_2^2 = 0$:  the line~$\ell$ is a triple line with tritangent plane $P_a$.
	
	\medskip
	\noindent\textit{In the proof of assertion~\eqref{i:cusp2},} To determine the nature of the singularity, after a linear change of variables in the coordinates $x_2, x_3, x_4$ only, we see that the fibre~$\cX_0$ is given by the two equations:
	\begin{align*}
		\overline\psi_{3, 0}(p_1) &= q_1(p_{1, 2}, p_{1, 3}, p_{1, 4})  + p_{1, 5} ^ 3 - k_1(p_{1, 2}, p_{1, 3}, p_{1, 4}),\\
		\overline\psi_{2, 1}(p_1) &=  - h_1(p_{1, 2}, p_{1, 3}, p_{1, 4}) + q_2(p_{1, 2}, p_{1, 3}, p_{1, 4}).
	\end{align*}
	Since $h_1(x_2, x_3, x_4) = x_2$, using the second equation we may express $p_{1, 2}$ locally at the origin as a holomorphic function $\tilde p_{1, 2}$ in the variables $p_{1, 3}, p_{1, 4}$ and we observe that its power expansion contains no linear term.  Replacing in the first equation, we see that the quadratic term is then equal to $q_1(0, p_{1, 3}, p_{1, 4})$. This means that the surface singularity in $\IC ^ 3$ of coordinates $(p_{1, 3}, p_{1, 4}, p_{1, 5})$ starts with a quadratic term  in the variables $p_{1, 3}, p_{1, 4}$ and a cubic term in $p_{1, 5}$. This quadratic term is
	nothing else than the intersection of the quadradic cone $q_1(x_2, x_3, x_4) = 0$ with the hyperplane $h_1(x_2, x_3, x_4) = 0$. We proved during the proof of assertion~\eqref{i:cusp1} that this intersection is the union of two distinct lines, so it is a quadratic form of rank two: this shows that $\cX_0$ has an $A_2$ singularity at the origin.
	
	\medskip
	\noindent\textit{Proof of assertion~\eqref{i:cusp4}.}
	By construction, $\sigma([\ell]) = [\iota(\ell)]$ for any $[\ell] \in \Grass 1 5$ and $\tau(x) = \iota(x)$ for any $x\in \Sigma$. For any $\xi \in \Hilb 2 \Sigma$, the image by $\iota$ of the plane $P = \Span(\cusp, \xi)$ 
	is $\iota(P) = \Span(\cusp, \tau^{[2]}(\xi))$ so $\varphi(\tau^{[2]}(\xi)) = \sigma(\varphi(\xi))$. This shows that $\varphi$~commutes with the actions induced by~$\iota$. Since $\tau^{[2]}$ acts nonsymplectically on~$\Hilb 2 \Sigma$ and since we defined the symplectic form on the smooth locus of $\Fano Y$ by pulling back of those of $\Hilb 2 \Sigma$, we deduce that $\sigma$ acts nonsymplectically on $\Fano Y \setminus \Sigma$. In our local Pl\"ucker coordinates, $\sigma$ acts by $p_{j,k}  \mapsto p_{j,k}$ and $p_{j, 5} \mapsto \xi p_{j, 5}$ for $k=2, 3, 4$ and $j = 0, 1$. The blowup relations $p_{1, i} = \frac{a_i}{a_5} p_{1, 5}$ impose that the automorphism~$\widetilde \sigma$ on~$\Blow \Sigma {\Fano Y}$ that makes~$\rho$ equivariant is defined by putting $a_5\mapsto \xi a_5$. Since $\rho$~is an equivariant symplectic resolution, 
	$\widetilde \sigma$~acts nonsymplectically. For any point $a \in \rho^{-1}{[\ell_0]}$ in an exceptional fiber, since $\widetilde\sigma$ sends $a_5$ to $\xi a_5$, we have $\iota(P_a) = P_{\widetilde\sigma (a)}$ so  $\widetilde\varphi^{-1}$~is equivariant on the exceptional locus.
\end{proof}

\begin{corollary}
	Under the assumptions of Theorem~\ref{th:Fano_cuspidal_cyclic_cubic}, if $Y$ contains no plane through $\cusp$, then generically:
	\[
	\Pic {\Blow \Sigma {\Fano Y}} \cong \Pic{ \Hilb 2 \Sigma} \cong \langle 6\rangle \oplus A_2(-1).
	\]
	The divisor $\Psi$ has two reducible components that generate the factor $A_2(-1)$.
\end{corollary}

\begin{proof}
	The computation of the Picard groups is a direct consequence of Theorem~\ref{th:Fano_cuspidal_cyclic_cubic},
	using the results of~\cite[\S4.3]{BCScubic}. Since $\Pic{\Fano Y} \cong \langle 6\rangle$, the morphism~$\varphi$ has relative Picard rank~$2$, hence $\Psi$~has two irreducible components. Note that the  $A_2$-contraction~$\varphi$ is in agreement with~\cite[Proposition~6.10, Corollary~6.11]{BakkerLehn_global}.
\end{proof}

\begin{remark}\label{rem:fix_locus}
	The fix locus of $\tau$ on $\Sigma$ is the genus four curve $\cC_4$ given by the equation $x_5 = 0$. The fix locus of $\Hilb 2 \tau$ on $\Hilb 2 \Sigma$ is thus the surface defined by the Hilbert square $\Hilb 2 \cC_4$, which is nothing else than the nonsingular Chow quotient $\cC_4^{(2)}$ (see for instance~\cite[\S6.1]{BCSclass}). On the other side, the fix locus of~$\iota$ on $Y$ is its cusp $\cusp$ and the cubic threefold $C$. The fix locus of~$\sigma$ on~$\Fano Y$ is thus the surface~$\Fano C$, which is a nonnormal surface singular along the lines contained in $C$ and passing through its node: similarly as above, this locus is isomorphic to $\cC_4$ and we have $\Fano C \cap \Sigma = \cC_4$. 
	Using the same computations as above, it is easy to see that $\Fano C \setminus \cC_4$ is isomorphic to the open subset of $\cC_4^{(2)}$ obtained by removing its intersection with the divisor $\Psi$ defined in Equation~\eqref{trident_divisor}.
	The equivariant diagram proven in Theorem~\ref{th:Fano_cuspidal_cyclic_cubic} shows that, when~$Y$ contains no plane through is cusp, the fix locus of $\tilde \sigma$ on $\BL_\Sigma {\Fano Y}$ is the strict transform $\widetilde{\Fano C}$ of $\Fano C$ by $\rho$ and the equivariant isomorphism $\widetilde{\varphi^{-1}}$ induces
	an isomorphism $\widetilde{\Fano C} \cong \cC_4^{(2)}$.
\end{remark}

\subsection{Nodal degenerations of cubic threefolds and nonseparated limits}

\label{ss:degeneration}

Let us give an interpretation of Theorem~\ref{th:Fano_cuspidal_cyclic_cubic} in the context of nodal generations of cubic threefolds as studied in~\cite[\S4]{BCScubic}. We recall briefly the context, refering to the paper for more details. The Fano variety of lines of a smooth cubic fourfold is a $6$-polarized 
IHS manifold deformation equivalent to the Hilbert square of a K3 surface (see~\cite{BeauvilleDonagi}).  The second cohomology space of these varieties, equipped with the Beauville--Bogomolov--Fujiki bilinear form, is isometric to the lattice:
\[
L \coloneqq U^{\oplus 3} \oplus E_8(-1)^{\oplus 2} \oplus \langle -2\rangle.
\]
We denote by $\cM_{L}^\circ$ a connected component of the moduli space of marked IHS manifolds belonging to this deformation class, with its surjective period map $\cP_L\colon \cM_L^\circ \to \Omega_L$ as in~\cite[Theorem~8.1]{Huybrechts}.
Consider a one parameter family $\{C_t\}_{t\neq 0}$ of smooth cubic threefolds in $\IP^4$, degenerating to a nodal cubic $C_0$. We consider the family of cubic fourfolds $Y_t$ obtained by triple covering of $\IP^4$ branched over $C_t$ and the associated family of Fano variety of lines $\Fano {Y_t}$, equipped with their order three nonsymplectic automorphism $\sigma_t$ induced by the covering automorphism. If $t\neq 0$, the pairs $(\Fano {Y_t}, \sigma_t)$ live in a moduli space $\cM_{6}^{\rho, \xi}$, where $\rho\in \orth(L)$ is the isometry class defined by $\sigma_t^\ast$ up to the choice of a marking and $\xi$ is the primitive third root of unity such that ${\sigma_t}\left|_{\HH^{2, 0}(\Fano{Y_t})}\right. = \xi \id$ 
(we refer to~\cite[\S3.5]{BCScubic} for the details on the adequate lattice-theoretical choices needed to fix the algebraic data). The restriction of the period map $\cP_L$ gives a surjective period map $\cP_6^{\rho, \xi} \colon \cM_6^{\rho, \xi}\to \Omega_6^{\rho, \xi}\setminus\cH$, where $\Omega_6^{\rho, \xi}$ is isomorphic to a $10$-dimensional complex ball and $\cH$ is a hyperplane arrangement. We have a commutative diagram:
\[
\xymatrix{
	\cM_6^{\rho, \xi}\ar[d]^{\cP_6^{\rho, \xi}} \ar@{^(->}[r]& \cM_L^\circ \ar@{->>}[d]^{\cP_L}\\
	\Omega_6^{\rho, \xi}\ar@{^(->}[r]& \Omega_L 
}
\]

We assume as above that the cubic $C_0$ has one single ordinary double point, so that the fourfold $Y_0$ has one single cusp. The locus of lines through the cusp defines a smooth K3 surface $\Sigma\subset \Fano {Y_0}$ with an order three nonsymplectic automorphism~$\tau$ induced by the covering automorphism.
The limit of the period points of the pairs $(\Fano{Y_t}, \sigma_t)$, when $t$ goes to zero, belongs to~$\cH$ (see~\cite[\S4.1]{BCScubic}):
\[
\lim\limits_{t\to 0} \cP_6^{\rho, \xi}(\Fano{Y_t}, \sigma_t) = \omega_0\in \cH.
\]

The proof of~\cite[Proposition~4.6]{BCScubic} shows that the choice of the IHS manifold~$\Sigma^{[2]}$ in $\cM_L^\circ$ over the period point~$\omega_0$, equipped with the automorphism~$\tau^{[2]}$, permits to extend holomorphically the period map $\cP_6^{\rho, \xi}$ over the hyperplane parametrizing the nodal degenerations. For this, it is necessary to change the representation $\rho$ as in~\cite[Proposition~4.4]{BCScubic} since the invariant lattice of the automorphism, that is $\langle 6\rangle$ when $t\neq 0$, becomes after this degeneration:
\[
U(3)\oplus \langle -2\rangle \cong \langle 6\rangle \oplus A_2(-1)
\]
(see~\cite[Lemma~4.3]{BCScubic}). The choice of the pair $\left(\Sigma^{[2]}, \tau^{[2]}\right)$ was motivated by Hodge theoretic arguments, based on the work of Allcock--Carlson--Toledo, but it has the little drawback that it hides the geometry of the singularity behind the abstract isometry of lattices above. Theorem~\ref{th:Fano_cuspidal_cyclic_cubic} shows that the blowup of $\Fano {Y_0}$ along $\Sigma$, equipped with the automorphism $\tilde \sigma$, is a more intuitive choice since $\Fano{Y_0}$ has a transversal $A_2$-singularity along $\Sigma$. If $Y_0$ contains no plane through its cusp, both choices are equivariantly isomorphic so the change is only for esthetic reasons. If $Y_0$~contains some planes through its cusp, the K3 surface $\Sigma$ contains some lines and the pairs $\left(\BL_\Sigma\Fano {Y_0}, \tilde\sigma\right)$ and 
$\left(\Sigma^{[2]}, \tau^{[2]}\right)$ are equivariantly birational to each other, so they are nonseparated points for the topology of $\cM_L^\circ$
(see~\cite[Theorem~4.7]{Huybrechts_birational}, \cite[\S4.4]{HuybrechtsHK}): there is no harm to choose one or the other. Our slogan may be summarized as:
\[
\left(\Sigma^{[2]}, \tau^{[2]}\right) = \lim\limits_{t\to 0} \left(\Fano{Y_t}, \sigma_t\right) = \left(\BL_\Sigma\Fano {Y_0}, \tilde\sigma\right).
\]

From the point of view of~\cite[\S5]{BCSball}, the period map of the moduli space parametrising such pairs of IHS fourfolds with an order three automorphism as above is not injective and each choice of a limit point corresponds to the choice of a different choice of chamber of $K(T)$-generality. Assuming that the hyperplane section $x_5 = 0$ is an irreducible genus four curve, this situation can happen only if $\Sigma$ contains at least three lines permuted by the automorphism, for the following reason. The automorphism~$\tau$ fixes pointwise its hyperplane section $x_5 = 0$.
Would $\Sigma$ contain an invariant line (this happens for instance when $\Sigma$ has only one or two lines), it would be globally invariant. Either this line would contain two isolated fix points or it would be pointwise fixed, both are not possible.

\bibliographystyle{amsplain}
\bibliography{Biblio}

\end{document}